\documentclass{amsart}
\usepackage{amstext,amssymb,amsthm,amsopn,newlfont,graphpap,graphics,graphicx,mathrsfs,enumitem}\usepackage{amstext,amssymb,amsthm,amsopn,newlfont,graphpap,graphics,graphicx,mathrsfs,enumitem}
\allowdisplaybreaks
\usepackage[parfill]{parskip}
\usepackage[noadjust]{cite}
\usepackage{epigraph}
\usepackage[colorlinks=true,
            linkcolor=red,
            urlcolor=blue,
            citecolor=magenta]{hyperref}
\usepackage{color}
\usepackage{mathrsfs}
\allowdisplaybreaks
\theoremstyle{plain}
\newtheorem{thm}{Theorem}[section]
\newtheorem*{thm*}{Theorem}

\newtheorem*{prop*}{Proposition}
\newtheorem{cor}{Corollary}[section]
\newtheorem*{cor*}{Corollary}
\newtheorem{lem}{Lemma}[section]
\newtheorem*{lem*}{Lemma}
\theoremstyle{definition}
\newtheorem{defn}{Definition}[section]
\newtheorem*{defn*}{Definition}

\newtheorem*{exmp*}{Example}
\newtheorem{exmps}{Examples}[section]
\newtheorem*{exmps*}{Examples}

\newtheorem{rem}{Remark}[section]
\newtheorem*{rem*}{Remark}
\newtheorem{rems}{Remarks}[section]
\newtheorem*{rems*}{Remarks}

\newtheorem*{note*}{Note}
\newcommand{\N}{{\mathbb N}}
\newcommand{\Z}{{\mathbb Z}}
\newcommand{\R}{{\mathbb R}}
\newcommand{\C}{{\mathbb C}}
\newcommand{\F}{{\mathbb F}}

\begin{document}
\title[general construct
of chaotic unbounded linear operators]
{On general construct\\
of chaotic unbounded linear operators\\
in Banach Spaces with Schauder Bases}
\author[Marat V. Markin]{Marat V. Markin}
\address{
Department of Mathematics\newline
California State University, Fresno\newline
5245 N. Backer Avenue, M/S PB 108\newline
Fresno, CA 93740-8001
}
\email{mmarkin@csufresno.edu}
\begin{abstract}
We utilize the idea underlying the construct of the classical weighted backward shift Rolewicz's operators to furnish a straightforward approach to a general construct of chaotic unbounded linear operators in a (real or complex) Banach space with a Schauder basis.
\end{abstract}
\maketitle

\section[Introduction]{Introduction}

The notions of \textit{hypercyclicity} and \textit{chaoticity}, traditionally considered and well studied for \textit{continuous} linear operators on Fr\'echet spaces, in particular for \textit{bounded} linear operators on Banach spaces, and shown to be a purely infinite-dimensional phenomena (see, e.g., \cite{Grosse-Erdmann-Manguillot,Guirao-Montesinos-Zizler,Rolewicz1969}), in \cite{B-Ch-S2001,deL-E-G-E2003} are extended 
to \textit{unbounded} linear ope\-rators in Banach spaces, where also found are sufficient conditions for unbounded hypercyclicity and certain examples of hypercyclic and chaotic unbounded linear differential operators.

\begin{defn}[Hypercyclicity and Chaoticity \cite{B-Ch-S2001,deL-E-G-E2003}] Let
\[
A:X\supseteq D(A)\to X
\]
be a (bounded or unbounded) linear operator in a (real or complex) Banach space $(X,\|\cdot\|)$. A nonzero vector 
\begin{equation}\label{idv}
x\in C^\infty(A):=\bigcap_{n=0}^{\infty}D(A^n)
\end{equation}
($A^0:=I$, $I$ is the \textit{identity operator} on $X$)
is called \textit{hypercyclic} if its orbit under $A$
\[
\left\{A^nx\right\}_{n\in\Z_+}
\]
($\Z_+:=\left\{0,1,2,\dots\right\}$ is the set of nonnegative integers) is dense in $X$.

Linear operators possessing hypercyclic vectors are
said to be \textit{hypercyclic}.

If there exist an $N\in \N$ ($\N:=\left\{1,2,\dots\right\}$ is the set of natural numbers) and a vector $x\in D(A^N)$ such that $A^Nx = x$, such a vector is called a \textit{periodic point} for the operator $A$.

Hypercyclic linear operators with a dense set of periodic points are said to be \textit{chaotic}.
\end{defn}

\begin{samepage}
\begin{rems}\
\begin{itemize}
\item In the definition of hypercyclicity, the underlying space must necessarily be \textit{separable}.
\item For a hypercyclic linear operator $A$, the subspace $C^\infty(A)$ (see \eqref{idv}), which contains the dense orbit of a hypercyclic vector, is also \textit{dense}.
\item As is easily seen, a periodic point $x$ of $A$, if any, also belongs to $C^\infty(A)$.
\end{itemize} 
\end{rems} 
\end{samepage}

In \cite{Markin2018(7)}, (bounded or unbounded) \textit{scalar type spectral operators} in a complex Banach, in particular \textit{normal} ones in a complex Hilbert space, (see, e.g., \cite{Dunford1954,Survey58,Dun-SchI,Dun-SchII,Dun-SchIII,Plesner}) are proven to be \textit{non-hypercyclic}. 

In \cite{Rolewicz1969}, S. Rolewicz gave the first example of a hypercyclic bounded linear operator on a Banach space (see also \cite{Grosse-Erdmann-Manguillot,Guirao-Montesinos-Zizler}), which on (real or complex) sequence space $l_p$ ($1\le p<\infty$) or $c_0$ (of vanishing sequences), the latter equipped with the supremum norm
\[
c_0\ni x:=(x_k)_{k\in \N}\mapsto \|x\|_\infty:=\sup_{k\in \N}|x_k|,
\] 
is the following weighted backward shift:
\[
A(x_k)_{k\in \N}:=\lambda(x_{k+1})_{k\in \N},
\]
where $\lambda\in \R$ or $\lambda\in \C$ with $|\lambda|>1$ is arbitrary. Rolewicz's operators also happen to be \textit{chaotic} (see, e.g., \cite[Example $2.32$]{Grosse-Erdmann-Manguillot}).

In \cite[Theorem $2.1$]{Markin2018(9)}, Rolewicz's example is modified to produce hypercyclic unbounded linear operators in $l_p$ ($1\le p<\infty$) or $c_0$ as follows.

\begin{thm}[{\cite[Theorem $2.1$]{Markin2018(9)}}]\label{TH}
In the (real or complex) sequence space $X:=l_p$ ($1\le p<\infty$) or $X:=c_0$, the weighted backward shift
\[
D(A)\ni x:=(x_k)_{k\in \N} \mapsto Ax:=(\lambda^k x_{k+1})_{k\in \N},
\]
where $\lambda\in \R$ or $\lambda\in \C$, $|\lambda|>1$ is arbitrary, with the domain
\[
D(A):=\left\{(x_k)_{k\in \N}\in X\,\middle|\,(\lambda^k x_{k+1})_{n\in \N}\in X \right\}
\]
is a hypercyclic unbounded linear operator.
\end{thm}

We are to generalize the above construct of unbounded weighted backward shifts to Banach spaces with Schauder bases (see, e.g., \cite{Lyust-Sob,Markin2018EFA}) and show, in particular, that the operators in the prior theorem are \textit{chaotic}.

\section[Preliminaries]{Preliminaries}

Here, for the reader's convenience, we outline certain essential preliminaries concerning Banach spaces with Schauder bases.

\begin{defn}[Schauder Basis of a Banach Space]
A \textit{Schauder basis} (also a \textit{countable basis}) of a (real or complex) Banach space $(X,\|\cdot\|)$ is a \textit{countably infinite set} of elements $\left\{e_n\right\}_{n\in \N}$ in $X$ such that
\begin{equation*}
\forall\, x\in X\ \exists! (x_k)_{k\in\N}\subseteq \F:\ x=\sum_{k=1}^\infty x_ke_k,
\end{equation*}
where $\F:=\R$ or $\F:=\C$, the series called the \textit{Schauder expansion} of $x$ and the numbers $x_k$, $k\in\N$, are called the \textit{coordinates} of $x$ relative to the Schauder basis $\left\{e_n\right\}_{n\in \N}$ (see, e.g., \cite{Lyust-Sob,Markin2018EFA}). 
\end{defn}

\begin{exmps}\
\begin{enumerate}[label={\arabic*.}]
\item The set $\left\{e_n:=\left\{\delta_{nk}\right\}_{k=1}^\infty \right\}_{n\in \N}$, where $\delta_{nk}$ is the \textit{Kronecker delta}, is 
a \textit{Schauder basis} for $l_p$ ($1\le p<\infty$) and $(c_0,\|\cdot\|_\infty)$ ($c_0$ is the space of vanishing sequences), the latter equipped with the supremum norm
\[
c_0\ni x:=(x_k)_{k\in \N}\mapsto \|x\|_\infty:=\sup_{k\in \N}|x_k|,
\]
with
\[
x:=(x_k)_{k\in\N}=\sum_{k=1}^\infty x_ke_k,
\]
\item The set 
$\left\{e_n\right\}_{n\in \Z_+}$, where
$e_0:=(1,1,\dots)$ is a \textit{Schauder basis} for $(c,\|\cdot\|_\infty)$ ($c$ is the space of convergent sequences) with
\[
x:=(x_k)_{k\in\N}=\sum_{k=0}^\infty x_ke_k,
\]
where
\[
x_0=\lim_{k\to\infty}x_k.
\]
\item Other examples of Banach spaces with more intricate Schauder bases are
$L_p(a,b)$ ($1\le p<\infty$) and
$C[a,b]$ ($-\infty<a<b<\infty$), the latter equipped with the maximum norm
\[
C[a,b]\ni x\mapsto \|x\|_\infty:=\max_{a\le t\le b}|x(t)|
\]
(see, e.g., \cite{Lyust-Sob,SingerI}).
\end{enumerate} 
\end{exmps} 

\begin{rems}\
\begin{itemize}
\item Observe that the aforementioned Schauder bases for the spaces $l_p$ ($1\le p<\infty$), $(c_0,\|\cdot\|_\infty)$, and $(c,\|\cdot\|_\infty)$
consist of \textit{unit vectors}.

Always, without loss of generality, one can regard that, for a Schauder basis $\{e_n\}_{n\in \N}$ of a Banach space 
$(X,\|\cdot\|)$,
\begin{equation}\label{normalized1}
\|e_n\|=1,\ n\in \N,
\end{equation}
which immediately implies that, for any
\[
x=\sum_{k=1}^\infty x_ke_k,
\]
by the convergence of the series,
\[
|x_k|=\|x_ke_k\|\to 0,\ k\to\infty.
\]
\item In particular, an orthonormal basis of a separable Hilbert space is a Schauder basis
satisfying condition \eqref{normalized1}
(see, e.g., \cite{Akh-Glaz}).
\item Every Banach space $(X,\|\cdot\|)$ with a Schauder basis is \textit{separable}, a countable dense subset being
\begin{equation}\label{cds}
Y:=\left\{x:=\sum_{k=1}^n x_ke_k\,\middle|\, n\in \N,\ x_1,\dots,x_n\ \text{are rational/complex rational}\right\}.
\end{equation}

The \textit{basis problem} on whether every separable Banach space has a Schauder basis posed by Stefan Banach was negatively answered by Per Enflo in \cite{Enflo1973}.
\item The profound fact useful for our discourse is that, for each
$n\in \N$, the \textit{Schauder coordinate functional}
\[
X\ni x=\sum_{k=1}^\infty x_ke_k\mapsto
x_n(x):=x_n\in \F
\]
is well defined, \textit{linear}, and \textit{bounded} (see, e.g., \cite{Lyust-Sob,Markin2018EFA}). 

Hence, if
\[
X\ni x^{(n)}=\sum_{k=1}^\infty x_k^{(n)}e_k\to 
x=\sum_{k=1}^\infty x_ke_k,\ n\to \infty,\
\text{in}\ (X,\|\cdot\|),
\]
then, for every $k\in \N$,
\[
x_k^{(n)}\to x_k,\ n\to \infty,
\]
the converse statement not being true (see, e.g., \cite{Markin2018EFA}).
\end{itemize} 
\end{rems} 

\section[Weighted Backward Shift Operators in Banach Spaces with Schauder Bases]{Weighted Backward Shift Operators\\ in Banach Spaces with Schauder Bases}

In this section, we analyze the nature of generalized backward shift operators in Banach spaces with Schauder bases.

\begin{lem}\label{Lem1} 
In the (real or complex) Banach space $(X,\|\cdot\|)$ with a Schauder basis $\left\{ e_k\right\}_{k\in\N}$ such that 
\begin{equation}\label{normalized}
\|e_k\|=1,\ k\in \N,
\end{equation}
the weighted backward shift
\[
D(A)\ni x=\sum_{k=1}^\infty x_ke_k \mapsto Ax:=\sum_{k=1}^\infty w_k x_{k+1}e_k,
\]
where $(w_k)_{k\in \N}$ is an arbitrary real- or complex-termed weight sequence such that
\begin{equation}\label{absin}
1\le |w_{k}|\le |w_{k+1}|,\ k\in \N,
\end{equation}
with the domain
\[
D(A):=\left\{x=\sum_{k=1}^\infty x_ke_k\in X\,\middle|\,\sum_{k=1}^\infty w_k x_{k+1}e_k\in X \right\}
\]
is a linear operator with a dense subspace 
$C^\infty(A)$ (see \eqref{idv}) and such that, for each $n\in \N$, the operator $A^n$ is closed and, provided
\begin{equation}\label{unb}
\lim_{k\to\infty}|w_k|=\infty,
\end{equation}
is unbounded.
\end{lem}

\begin{proof}\quad The \textit{linearity} of $A$ is obvious.

For each $n\in \N$,
\[
A^nx=\sum_{k=1}^\infty \left[\prod_{j=k}^{n+k-1}w_j\right]x_{k+n}e_k,\
\]
\[
x=\sum_{k=1}^\infty x_ke_k\in D(A^n)
=\left\{x=\sum_{k=1}^\infty x_ke_k\in X\,\middle|\,\sum_{k=1}^\infty \left[\prod_{j=k}^{n+k-1}w_j\right]x_{k+n}e_k\in X \right\}.
\]
with
\begin{equation}\label{incl1}
D(A^{n+1})\subseteq D(A^n),\ n\in\N,
\end{equation}
the linear operator $A^n$ being densely defined since the domain $D(A^n)$ contains the countable dense in $(X,\|\cdot\|)$ subset $Y$ (see \eqref{cds}) and hence, so does the subspace 
\[
C^\infty(A):=\bigcap_{n=0}^{\infty}D(A^n)
\]
of all possible initial values for the orbits under $A$, which is also \textit{dense} in $(X,\|\cdot\|)$.

Let
\[
D(A^n)\ni x^{(m)}=\sum_{k=1}^\infty x_k^{(m)}e_k\to 
x=\sum_{k=1}^\infty x_ke_k,\ m\to \infty,
\]
and
\[
A^nx^{(m)}=\sum_{k=1}^\infty \left[\prod_{j=k}^{n+k-1}w_j\right]x_{k+n}^{(m)}e_k
\to y=\sum_{k=1}^\infty y_ke_k,\ m\to \infty,
\]
then (see Preliminaries), for each $k\in\N$,
\[
x_k^{(m)}\to x_k,\ m\to\infty,
\]
and
\[
\left[\prod_{j=k}^{n+k-1}w_j\right]x_{k+n}^{(m)}\to y_k,\ m\to\infty.
\]

Whence, we infer that
\[
y_k=\left[\prod_{j=k}^{n+k-1}w_j\right]x_{k+n},\ m\in \N,
\]
and therefore,
\[
x\in D(A^n)\ \text{and}\ y=Ax,
\]
which implies that, for each $n\in \N$, the operator $A^n$ is \textit{closed} (see, e.g., \cite{Dun-SchI,Markin2018EFA}).

In view of \eqref{normalized}, with the weight sequence $(w_k)_{k\in \N}$ satisfying 
conditions \eqref{absin} and \eqref{unb}, for every $n\in\N$, the operator $A^n$ is  \textit{unbounded} since
\[
\|A^ne_{2n}\|=\left\|\left[\prod_{j=n}^{2n-1}w_j\right]e_{n}\right\|=\left[\prod_{j=n}^{2n-1}|w_j|\right]\|e_n\|\ge |w_n|\to \infty,\ n\to \infty.
\]
\end{proof}

\begin{rem}
Observe that, under restriction \eqref{normalized}, conditions \eqref{absin} and \eqref{unb} are \textit{sufficient} for the unboundedness of 
$A^n$ ($n\in \N$), but, generally, may not be necessary.
\end{rem}
\section[General Construct of Chaotic Unbounded Linear Operators in Banach Spaces with Schauder Bases]{General Construct of Chaotic Unbounded Linear Operators\\ in Banach Spaces with Schauder Bases}

Here, we furnish a straightforward approach to a general construct of chaotic unbounded linear operators in a (real or complex) Banach space with a Schauder basis utilizing the idea, which underlies the construct of the classical weighted backward shift Rolewicz's operators, i.e., the denseness in the spaces $l_p$ ($1\le p<\infty$) and $(c_0,\|\infty\|_\infty)$ of the subspace
$c_{00}$ of \textit{eventually zero} sequences.
In a Banach space $(X,\|\cdot\|)$ with a Schauder basis $\{e_n\}_{n\in\N}$ the natural analogue of such a subspace is that of \textit{``eventually zero"} vectors relative to 
$\{e_n\}_{n\in\N}$:
\[
\left\{x:=\sum_{k=1}^\infty x_ke_k\in X\,\middle|\, \exists\, n\in \N:\ x_k=0,\ k>n\right\}.
\]

Thus, in the spaces $l_p$ ($1\le p<\infty$) and $(c_0,\|\infty\|_\infty)$ relative to their standard Schauder bases described in Preliminaries, this subspace merely coincides with $c_{00}$, whereas in the space $(c,\|\cdot\|_\infty)$ relative to its standard Schauder basis described in Preliminaries, this subspace consists of all \textit{eventually constant} sequences.

\begin{thm}\label{Thm1} 
In the (real or complex) Banach space $(X,\|\cdot\|)$ with a Schauder basis $\left\{ e_k\right\}_{k\in\N}$ such that 
\begin{equation}\label{normalized2}
\|e_k\|=1,\ k\in \N,
\end{equation}
the weighted backward shift
\[
D(A)\ni x=\sum_{k=1}^\infty x_ke_k \mapsto Ax:=\sum_{k=1}^\infty w_k x_{k+1}e_k,
\]
where $(w_k)_{k\in \N}$ is an arbitrary real- or complex-termed weight sequence satisfying
\eqref{absin} and such that
\begin{equation}\label{serconv}
\sum_{k=1}^\infty \left|w_k\right|^{-1}<\infty,
\end{equation}
with the domain
\[
D(A):=\left\{x=\sum_{k=1}^\infty x_ke_k\in X\,\middle|\,\sum_{k=1}^\infty w_k x_{k+1}e_k\in X \right\}
\]
is a chaotic unbounded linear operator.
\end{thm}

\begin{proof}\quad 
Observe that, since condition \eqref{serconv},
implies condition \eqref{unb}, and hence, by Lemma \ref{Lem1}, $A$ is a linear operator with a dense subspace $C^\infty(A)$ (see \eqref{idv}) and such that, for each $n\in \N$, the operator $A^n$ is \textit{closed} and \textit{unbounded}.

We proceed in a similar fashion as in the proof of \cite[Theorem $2.1$]{Markin2018(9)} (cf. \cite[Theorem $1$]{Rolewicz1969} and \cite[Example $2.18$]{Grosse-Erdmann-Manguillot}). 

The right inverse of $A$
\begin{equation}\label{rightinv}
Bx:=\sum_{k=1}^\infty w_k^{-1}x_ke_{k+1},
\end{equation}
is linear operator well defined on the entire space $X$ since, 
for any 
\[
x=\sum_{k=1}^\infty x_ke_k\in X,
\]
in view of \eqref{normalized2},
\[
|x_k|=\|x_ke_k\|\to 0,\ k\to\infty,
\]
(see Preliminaries), and hence, by the estimate
\[
\left\|w_k^{-1}x_ke_{k+1}\right\|
=|w_k|^{-1}|x_k|\|e_{k+1}\|\le \left[\sup_{j\in\N}|x_j|\right]\left|w_k\right|^{-1},\ k\in\N,
\]
and, in view of condition \eqref{serconv}, the series in \eqref{rightinv} converges.

Further, for any $n\in \N$ we have:
\[
B^nx=\sum_{k=1}^\infty  \left[\prod_{j=k}^{n+k-1}w_j^{-1}\right]x_ke_{k+n},\ x=\sum_{k=1}^\infty x_ke_k\in X.
\]

Let
\begin{equation}\label{def2}
\left\{y^{(m)}:=\sum_{k=1}^\infty y^{(m)}_ke_k\right\}_{m\in \N}
\end{equation}
be an arbitrary enumeration of the countable dense subset $Y$ (see \eqref{cds}) with
\begin{equation}\label{def1}
k_m:=\max\left\{k\in \N\,\middle|\,y^{(m)}_k\neq 0 \right\},\ m\in\N.
\end{equation}

Inductively, one can choose an increasing sequence
$(n_m)_{m\in \N}$ of natural numbers such that, for all $j,m\in \N$ with $m>j$,
\begin{equation}\label{est1}
n_m-n_j\ge \max(m,k_j)\ \text{and}\
\prod_{i=1}^{n_m}|w_i|
\ge 
\left[\prod_{i=n_m-n_j+1}^{n_m}|w_i|\right]|w_m| \sum_{k=1}^{k_m}\left|y_k^{(m)}\right|.
\end{equation}

Let us show that the vector
\begin{equation*}
x:=\sum_{j=1}^\infty B^{n_j}y^{(j)}
\end{equation*}
is \textit{hypercyclic} for $A$, the series converging in $(X,\|\cdot\|)$
by condition \eqref{serconv} since, for any $j=2,3,\dots$, in view of \eqref{normalized2}, \eqref{absin}, and \eqref{est1},
\begin{multline*}
\left\|B^{n_j}y^{(j)}\right\|
=\left\|\sum_{k=1}^\infty  \left[\prod_{i=k}^{n_j+k-1}w_i^{-1}\right]y_k^{(j)}e_{k+n_j}\right\|
\le \sum_{k=1}^{k_j}\left[\prod_{i=k}^{n_j+k-1}\left|w_i\right|^{-1}\right] \left|y_k^{(j)}\right|\|e_{k+n_j}\|
\\
\ \ \
\le \left[\prod_{i=1}^{n_j}\left|w_i\right|^{-1}\right] \sum_{k=1}^{k_j}\left|y_k^{(j)}\right|\le \left|w_j\right|^{-1}.
\hfill
\end{multline*}

Further, for each $m\in \N$,
\begin{multline*}
\sum_{j=1}^\infty A^{n_m}B^{n_j}y^{(j)}
=\sum_{j=1}^{m-1} A^{n_m-n_j}y^{(j)}+y^{(m)}
+\sum_{j=m+1}^\infty B^{n_j-n_m}y^{(j)}
\\
\hfill
\text{since, by \eqref{est1}, for $j=1,\dots,m-1$, $n_m-n_j\ge k_j$, by \eqref{def1}, $A^{n_k-n_j}y^{(j)}=0$;}
\\
\ \ \
=y^{(m)}+\sum_{j=m+1}^\infty B^{n_j-n_m}y^{(j)}.
\hfill
\end{multline*}

Since, for all $m,j\in \N$ with $j\ge m+1$, in view of \eqref{normalized2}, \eqref{absin}, and \eqref{est1},
\begin{multline}\label{est2}
\left\|B^{n_j-n_m}y^{(j)}\right\|
=\left\|\sum_{k=1}^\infty  \left[\prod_{i=k}^{n_j-n_m+k-1}w_i^{-1}\right]y_k^{(j)}e_{k+n_j-n_m}\right\|
\\
\shoveleft{
\le \sum_{k=1}^{k_j}\left[\prod_{i=k}^{n_j-n_m+k-1}\left|w_i\right|^{-1}\right] \left|y_k^{(j)}\right|\|e_{k+n_j-n_m}\|
\le \left[\prod_{i=1}^{n_j-n_m}\left|w_i\right|^{-1}\right] \sum_{k=1}^{k_j}\left|y_k^{(j)}\right|\le \left|w_j\right|^{-1}.
}\\
\end{multline}

This implies that, for any $m\in \N$, the series
\[
\sum_{j=1}^\infty A^{n_m}B^{n_j}y^{(j)}
\]
converges in $(X,\|\cdot\|)$, and hence, by the \textit{closedness} of $A^{n_m}$ and in view of inclusion \eqref{incl1},
\[
x\in \bigcap_{m=1}^{\infty}D(A^{n_m})=C^\infty(A)
\]
and 
\[
A^{n_m}x=y^{(m)}+\sum_{j=m+1}^\infty B^{n_j-n_m}y^{(j)},\ m\in \N.
\]

Furthermore, since, by \eqref{est2}, for each $k\in \N$,
\[
\left\|A^{n_m}x-y^{(m)}\right\|
\le \sum_{j=m+1}^\infty \left\|B^{n_j-n_m}y^{(j)}\right\|
\le \sum_{j=k+1}^\infty
\left|w_j\right|^{-1}\to 0,\ k\to\infty,
\] 
which, in view of the denseness of $Y$ in $(X,\|\cdot\|)$, shows that the orbit
$\left\{A^nx\right\}_{n\in\Z_+}$ of $x$ under $A$
is dense in $(X,\|\cdot\|)$ as well. This implies that the operator $A$ is hypercyclic.

Now, let us show that the operator $A$ is \textit{chaotic} (cf. \cite[Example $2.32$]{Grosse-Erdmann-Manguillot}). Indeed, for any $N\in \N$, let $x_1,\dots,x_N\in \F$ be arbitrary, then
\begin{equation}\label{pp}
x:=\sum_{m=1}^Nx_me_m+ \sum_{k=1}^{\infty}\left[\sum_{m=1}^N
\left[\prod_{i=(k-1)N+m}^{kN+m-1}w_i^{-1}\right]x_me_{kN+m}\right],
\end{equation}
the series converging in $(X,\|\cdot\|)$ by \eqref{serconv} since, for any $k\in \N$, in view of \eqref{normalized2} and \eqref{absin},
\begin{multline}\label{pp2}
\left\|\sum_{m=1}^N\left[\prod_{i=(k-1)N+m}^{kN+m-1}w_i^{-1}\right]x_me_{kN+m}\right\|
\le \sum_{m=1}^N \left[\prod_{i=(k-1)N+m}^{kN+m-1}\left|w_i\right|^{-1}\right]|x_m|\|e_{kN+m}\|
\\
\shoveleft{
\le \left[\sum_{m=1}^N|x_m|\right]\left|w_{kN}\right|^{-1},
}\\
\end{multline}
we have:
\[
x\in D(A^N)\ \text{and}\ A^N x=x.
\]

For any 
\[
y:=\sum_{m=1}^n y_me_m\in Y
\]
with some $n\in \N$ (see \eqref{cds}) and arbitrary
$N\ge n$, consider the periodic point defined by \eqref{pp} with
\[
x_m:=\begin{cases}
y_m,& m=1,\dots,n,\\
0,& m=n+1,\dots,N.
\end{cases}
\]

Then, in view of \eqref{pp2}, by condition \eqref{serconv},
\begin{multline*}
\|x-y\|=\left\|\sum_{k=1}^{\infty}\left[\sum_{m=1}^N\left[\prod_{i=(k-1)N+m}^{kN+m-1}w_i^{-1}\right]x_me_{kN+m}\right]\right\|
\\
\ \ \
\le
\left[\sum_{m=1}^n|y_m|\right]\sum_{k=1}^{\infty}\left|w_{kN}\right|^{-1} 
\le
\left[\sum_{m=1}^n|y_m|\right]\sum_{k=N}^{\infty}\left|w_{k}\right|^{-1}
\to 0,\ N\to \infty.
\hfill
\end{multline*}

Whence, in view of the denseness of $Y$ in $(X,\|\cdot\|)$, we infer that the set of periodic points of $A$ is dense in $(X,\|\cdot\|)$ as well,
which completes the proof.
\end{proof}

The exponential weight sequence
\[
w_n:=\lambda^n,\ n\in\N,
\]
with $\lambda\in \R$ or $\lambda\in \C$, $|\lambda|>1$ satisfying conditions \eqref{absin} and \eqref{serconv}, for the sequence spaces $l_p$ ($1\le p<\infty$) or $(c_0,\|\cdot\|_\infty)$, we arrive at the following corollary improving the result of Theorem \ref{TH}.

\begin{cor}\label{CC}
In the (real or complex) sequence space $l_p$ ($1\le p<\infty$) or $(c_0,\|\cdot\|_\infty)$, the weighted backward shift
\[
D(A)\ni x:=(x_k)_{k\in \N} \mapsto Ax:=(\lambda^k x_{k+1})_{k\in \N},
\]
where $\lambda\in \R$ or $\lambda\in \C$, $|\lambda|>1$ is arbitrary, with the domain
\[
D(A):=\left\{(x_k)_{n\in \N}\in X\,\middle|\,(\lambda^k x_{k+1})_{k\in \N}\in X \right\}
\]
is a chaotic unbounded linear operator.
\end{cor}

\begin{rem}
For a detailed discourse on the weighted shifts in Fr\'echet sequence spaces, see \cite[Chapter $4$]{Grosse-Erdmann-Manguillot}.
\end{rem}

\section{Acknowledgments}

I would like to express sincere gratitude to my graduate advisees, Mr. Edward Sichel and Mr. John Jimenez, Department of Mathematics, California State University, Fresno, for stimulating discussions and their keen interest in the subject matter.

 
\end{document}